\documentclass[12pt]{article}
\usepackage{amsmath,amssymb,amsbsy,amsfonts,amsthm,latexsym,
                    amsopn,amstext,amsxtra,euscript,amscd}

\newtheorem{thm}{Theorem}
\newtheorem{lem}[thm]{Lemma}

\newcommand{\be}{\begin{equation}}
\newcommand{\ee}{\end{equation}}
\newcommand{\benn}{\begin{equation*}}
\newcommand{\eenn}{\end{equation*}}

\renewcommand{\a}{\ensuremath{\alpha}}

\renewcommand{\(}{\left(}
\renewcommand{\)}{\right)}

\def\\{\cr}
\def\[{\left[}
\def\]{\right]}
\def\<{\langle}
\def\>{\rangle}

\def\cA{{\mathcal A}}

\def\cD{{\mathcal D}}
\def\cE{{\mathcal E}}
\def\cF{{\mathcal F}}
\def\cG{{\mathcal G}}
\def\cH{{\mathcal H}}
\def\cI{{\mathcal I}}
\def\cJ{{\mathcal J}}

\def\cM{{\mathcal M}}

\def\cS{{\mathcal S}}
\def\cT{{\mathcal T}}

\def\notdivides{\mathrel{\kern-3pt\not\!\kern3.5pt\bigm|}}

\begin{document}

\title{On the largest prime factor of the Mersenne numbers}

\author{
{\sc Kevin Ford}\\
{Department of Mathematics}\\
{The University of Illinois at Urbana-Champaign Urbana}\\
{Champaign, IL 61801, USA}\\
{ford@math.uiuc.edu} \and
{\sc Florian~Luca} \\
{Instituto de Matem{\'a}ticas}\\
{ Universidad Nacional Autonoma de M{\'e}xico} \\
{C.P. 58089, Morelia, Michoac{\'a}n, M{\'e}xico} \\
{fluca@matmor.unam.mx} \\
\\
{\sc Igor E.~Shparlinski} \\
{Department of Computing}\\
{Macquarie University} \\
{Sydney, NSW 2109, Australia} \\
{igor@ics.mq.edu.au}}

\date{}
\maketitle

\begin{abstract}
Let $P(k)$ be the largest prime factor of the positive integer $k$.
In this paper, we prove that the series
%%KF2
$$
\sum_{n\ge 1}\frac{(\log n)^{\a}}{P(2^n-1)}
$$
is convergent for each constant $\a<1/2$, which gives a more precise
form of a result of C.~L.~Stewart of 1977.
\end{abstract}

\section{Main Result}

Let $P(k)$ be the largest prime factor of the positive integer $k $.
The quantity $P(2^n-1)$ has been investigated by many authors
(see~\cite{BV,EKP,ES,MP,MW,Pom1,Sch,S1,S2}). For example, the best
known lower bound
$$P(2^n-1)\ge 2n+1, \qquad
{\text{for}}~n\ge 13$$ is due to Schinzel~\cite{Sch}. No better
bound is known even for all sufficiently large values of $n$.

C.~L.~Stewart~\cite{S1,S2} gave better bounds provided that $n$
satisfies certain arithmetic or combinatorial properties. For
example, he showed in~\cite{S2}, and this was also proved
independently by Erd\H os and Shorey in~\cite{ES}, that
$$
P(2^p-1)> cp \log p
$$
holds for all sufficiently large prime numbers $p$, where $c>0$ is
an absolute constant and $\log$ is the natural logarithm. This was
an improvement upon a previous result of his from~\cite{S1} with
$(\log p)^{1/4}$ instead of $\log p$. Several more results along
these lines are presented in Section~\ref{sec:motiv}.

Here, we continue to study $P(2^n-1)$ from a point of view familiar
to number theory which has not yet been applied to $P(2^n-1)$. More
precisely, we study the convergence of the series
\begin{equation}
\label{eq:ser} \sigma_\alpha = \sum_{n\ge 1}\frac{(\log
n)^\alpha}{P(2^n-1)}
\end{equation}
for some real parameter $\alpha$.

Our result is:

\begin{thm}
\label{thm:main} The series $\sigma_\alpha$ is convergent for all
$\alpha<1/2$.
\end{thm}

The rest of the paper is organized as follows. We introduce some
notation in Section~\ref{sec:notation}. In Section~\ref{sec:motiv},
we comment on why Theorem~\ref{thm:main} is interesting and does not
immediately follow from already known results. In
Section~\ref{sec:fund}, we present a result C.~L.~Stewart~\cite{S2}
which plays a crucial role in our argument. Finally, in
Section~\ref{sec:proof}, we give a proof of Theorem~\ref{thm:main}.

\section{Notation}
\label{sec:notation}

In what follows, for a positive integer $n$ we use $\omega(n)$ for
the number of distinct prime factors of $n$, $\tau(n)$ for the
number of divisors of $n$ and $\varphi(n)$ for the Euler function of
$n$. We use the Vinogradov symbols $\gg$, $\ll$ and $\asymp$ and the
Landau symbols $O$ and $o$ with their usual meaning. The constants
implied by them might depend on $\alpha$. We use the letters $p$ and
$q$ to denote prime numbers. Finally, for a subset $\cA$ of positive
integers and a positive real number $x$ we write $\cA(x)$ for the
set $\cA\cap [1,x]$.

\section{Motivation}
\label{sec:motiv}

In~\cite{S2}, C.~L.~Stewart proved the following two statements:

\begin{description}
\item[A.] If $f(n)$ is any positive real valued function which is
increasing and $f(n)\to\infty$ as $n\to\infty$, then the inequality
$$
P(2^n-1)>\frac{n(\log n)^2}{f(n)\log\log n}
$$
holds for all positive integers $n$ except for those in a set of
asymptotic density zero.

\item[B.] Let $\kappa<1/\log 2$ be fixed. Then the inequality
$$
P(2^n-1)\ge C(\kappa)  \frac{\varphi(n)\log n}{2^{\omega(n)}}
$$
holds for all positive integers $n$ with $\omega(n)<\kappa\log\log
n$, where $C(\kappa) > 0$ depends on $\kappa$.
\end{description}

Since for every fixed $\varepsilon>0$  we have
$$
\sum_{n\ge 2}\frac{\log \log n}{n(\log n)^{1+\varepsilon}} < \infty,
$$
the assertion {\bf A} above, taken with $f(n)=(\log
n)^{\varepsilon}$ for fixed some small positive $\varepsilon <
1-\alpha$, motivates our Theorem~\ref{thm:main}. However, since
C.~L.~Stewart~\cite{S2} gives no analysis of the exceptional set in
the assertion {\bf A} (that is,  of the size of the set of numbers
$n\le x$ such that the corresponding  estimate fails for a
particular choice of $f(n)$), this alone does not lead to a proof of
Theorem~\ref{thm:main}.

In this respect, given that the distribution of positive integers
$n$ having a fixed number of prime factors $K<\kappa\log\log n$ is
very well-understood starting with the work of Landau and continuing
with the work of Hardy and Ramanujan~\cite{HarRam}, it may seem that
the assertion {\bf B} is  more suitable  for our purpose. However,
this is not quite so either since most $n$ have
$\omega(n)>(1-\varepsilon)\log\log n$ and for such numbers the lower
bound on $P(2^n-1)$ given by {\bf B}  is only of the shape
$\varphi(n)(\log n)^{1-(1-\varepsilon)\log 2}$ and this is not
enough to guarantee the convergence of series~\eqref{eq:ser} even
with $\alpha = 0$.

Conditionally, Murty and Wang~\cite{MW} have shown   the
$ABC$-conjecture implies that $P(2^n-1)>n^{2-\varepsilon}$ for all
$\varepsilon>0$ once $n$ is sufficiently large with respect to
$\varepsilon$. This certainly implies the conditional convergence of
series~\eqref{eq:ser} for all fixed $\alpha>0$. Murata and
Pomerance~\cite{MP} have proved, under the Generalized Riemann
Hypothesis for various Kummerian fields, that the inequality
$P(2^n-1)>n^{4/3}/\log\log n$ holds for almost all $n$, but they did
not give explicit upper bounds on the size of the exceptional set
either.

\section{Main Tools}
\label{sec:fund}

As we have mentioned in Section~\ref{sec:motiv}, neither assertion
{\bf A} nor {\bf B} of Section~\ref{sec:motiv} are directly suitable
for our purpose.  However, another criterion, implicit in the work
of C.~L.~Stewart~\cite{S2} and which we present as
Lemma~\ref{lem:Div} below (see also Lemma~3 in~\cite{MP}), plays an
important role in our proof.

\begin{lem}
\label{lem:Div} Let $n\ge 2$, and let $d_1 < \cdots < d_\ell$ be all
$\ell = 2^{\omega(n)}$ divisors of $n$ such that $n/d_i$ is
square-free. Then  for all $n>6$,
$$
\# \{p \mid 2^n -1 \ : \ p\equiv 1\pmod n\} \gg
  \frac{
\displaystyle\log\( 2 + \frac{\Delta(n)}{\tau(n)} \)}{\log\log
P(2^n-1)},
$$
where
$$
\Delta(n) =  \max_{i=1, \ldots, \ell-1} d_{i+1}/d_i.
$$
\end{lem}

The proof of C.~L.~Stewart~\cite{S2} of Lemma~\ref{lem:Div}
uses the original lower bounds for linear forms in logarithms of
algebraic numbers due to Baker. It is interesting to notice that
following~\cite{S2} (see also~\cite[Lemma 3]{MP}) but using instead
the sharper lower bounds for linear forms in logarithms due to
E.~M.~Matveev~\cite{Matv}, does not seem to lead to any
improvement of  Lemma~\ref{lem:Div}.

Let $1=d_1<d_2<\cdots<d_{\tau(n)}=n$ be all the divisors of $n$
arranged in increasing order and let
$$
\Delta_0(n)=\max_{i\le \tau(n)-1} d_{i+1}/d_i.
$$
Note that $\Delta_0(n)\le \Delta(n)$.

We need the following result of E.~Saias~\cite{Sa} on the distribution
of positive integers $n$ with ``dense divisors''. Let
$$
\cG(x,z)=\{n\le x\ : \ \Delta_0(n)\le z\}.
$$

\begin{lem}
\label{lem:saias} The bound
$$
\#\cG(x,z)\asymp x\frac{\log z}{\log x}
$$
holds uniformly for $x\ge z\ge 2$.
\end{lem}

Next we address the structure of  integer with $\Delta_0(n) \le z$.
In what follows, as usual, an empty product is,  by convention,
equal to $1$.

\begin{lem}
\label{lem:Struct} Let $n=p_1^{e_1}\cdots p_k^{e_k}$ be the prime
number factorization of a positive integer $n$,  such that 
$p_1<\cdots<p_k$. Then $\Delta_0(n) \le z$
if and only if for each $i\le k$, the inequality
$$
p_i\le z\prod_{j<i} p_j^{e_j}
$$
holds.
\end{lem}

\begin{proof} The necessity is clear since otherwise the ratio
of the two  consecutive divisors
$$
\prod_{j<i} p_j^{e_j} \qquad \text{and} \qquad p_i
$$
is larger than $z$.

The sufficiency can be proved by induction on $k$. Indeed for $k=1$
it is trivial.  By the induction assumption, we also have $\Delta(m)
\le z$, where $m = n/p_1^{e_1}$. Remarking that $p_1 \le z$, we also
conclude that $\Delta(n) \le z$.
\end{proof}

\section{Proof of Theorem~\ref{thm:main}}
\label{sec:proof}

We put $\cE=\{n:\tau(n)\ge (\log n)^3\}$. To bound $\#\cE(x)$, let
$x$ be large and $n\le x$. We may assume that $n>x/(\log x)^2$ since
there are only at most $x/(\log x)^2$ positive integers $n\le
x/(\log x)^2$. Since $n\in \cE(x)$, we have that
$\tau(n)>(\log(x/\log x))^3>0.5(\log x)^3$ for all $x$ sufficiently
large. Since
$$
\sum_{n\le x}\tau (n)=O(x\log x)
$$
(see~\cite[Theorem~320]{HarWr}), we get that
$$
\#\cE(x)\ll \frac{x}{(\log x)^2}.
$$
By the Primitive Divisor Theorem (see~\cite{BV}, for example), there
exists a prime factor $p\equiv 1 \pmod n$ of $2^n-1$ for all $n>6$.
Then, by partial summation,
\begin{eqnarray*}
\sum_{n\in \cE(x)}\frac{(\log n)^\alpha}{P(2^n-1)}& \le & \sum_{n\in
\cE(x)}\frac{(\log n)^\alpha}{n}\le 1 + \int_2^x
    \frac{(\log t)^\alpha}{t} d{\#\cE(t)}\\
& \le &
1+\frac{\#\cE(x)}{x}+\int_2^x\frac{\#\cE(t) (\log t)^\alpha}{t^2}dt\\
& \ll &  1+\int_2^x \frac{dt}{t(\log t)^{2-\alpha}} \ll 1.
\end{eqnarray*}
Hence,
\begin{equation}
\label{eq:Sum E} \sum_{n\in \cE} \frac{(\log n)^{\alpha} }{P(2^n-1)}
< \infty.
\end{equation}

We now let $\cF=\{n:P(2^n-1)>n(\log n)^{1+\alpha}(\log\log n)^2\}.$
Clearly,
\begin{equation}
\label{eq:Sum F} \sum_{n\in \cF} \frac{(\log n)^{\alpha}
}{P(2^n-1)}\le \sum_{n\ge 1}\frac{1}{n \log n  (\log\log n)^2} <
\infty.
\end{equation}
  From now on, we assume that $n\not\in \cE\cup \cF$. For a given $n$,
we let
$$
\cD(n)=\{d: dn+1~{\text{\rm is~a~prime~factor~of}}~2^n-1\},
$$
and
$$
D^+(n) = \max\{d\in \cD(n)\}.
$$
Since $P(2^n -1) \ge d(n)n +1$, we have
\begin{equation}
\label{eq:bound d(n)} D^+(n)\le \(\log n\)^{1+\alpha}(\log\log n)^2.
\end{equation}

Further, we let $x_L=e^L$. Assume that $L$ is large enough. Clearly,
for $n\in [x_{L-1},x_L]$ we have  $D^+(n)\le L^{1+\alpha}(\log
L)^2$. We let $\cH_{d,L}$ be the set of $n\in [x_{L-1},x_L]$ such
that $D^+(n)=d$. We then note that by partial summation
\begin{equation}
\label{eq:SL}
\begin{split}
S_L & =  \sum_{\substack{x_{L-1}\le n\le x_L\\ n \not\in \cE\cup
\cF}}\frac{(\log n)^{\alpha}}{P(2^n-1)}\le  L^{\alpha}\sum_{d\le
L^{1+\alpha}(\log L)^2}\sum_{n\in \cH_{d,L}}\frac{1}{nd+1}\\ & <
\frac{L^{\alpha}}{x_{L-1}} \sum_{d\le L^{1+\alpha}(\log L)^2}
\frac{\#\cH_{d,L}}{d} \ll \frac{L^{\alpha}}{x_L} \sum_{d\le
L^{1+\alpha}(\log L)^2} \frac{\#\cH_{d,L}}{d}.
\end{split}
\end{equation}

We now   estimate $\#\cH_{d,L}$. We let $\varepsilon>0$ to be a
small positive number depending on $\alpha$ which is to be specified
later. We split $\cH_{d,L}$ in two subsets as follows:

Let $\cI_{d,L}$ be the set of $n\in \cH_{d,L}$ such that
$$
\#\cD(n)>\frac{1}{M}\(\log n\)^{\alpha+\varepsilon}(\log\log
n)^2>\frac{1}{M}L^{\alpha+\varepsilon}(\log L)^2,
$$
where $M=M(\varepsilon)$ is some positive integer depending on
$\varepsilon$ to be determined later. Since $D^+(n)\le
L^{1+\alpha}(\log L)^2$, there exists an interval of length
$L^{1-\varepsilon}$ which contains at least $M$ elements of
${\cD}(n)$. Let them be $d_0<d_1<\cdots<d_{M-1}$. Write
$k_i=d_i-d_0$ for $i=1,\ldots,M-1$. For fixed
$d_0,k_1,\ldots,k_{M-1}$, by the Brun sieve (see, for example,
Theorem~2.3 in~\cite{HR}),
\begin{equation}
\begin{split}
\label{eq:dd1d2} \#\{n   \in & [x_{L-1}, x_L]\ : \ d_in+1 \ \text{is
a prime for all}\
i=1,\ldots,M \}\\
&  \ll
  \frac{x_L}{(\log(x_L))^M} \prod_{p\mid
d_1\cdots d_M}\(1-\frac{1}{p}\)^{-M}\ll \frac{x_L}{L^M}
\(\frac{\prod_{i=1}^Md_i}{\varphi\(\prod_{i=1}^M
d_i\)}\)^M \\
& \ll  \frac{x_L(\log\log L)^M}{L^M},
\end{split}
\end{equation}
where we have used that  $\varphi(m)/m\gg 1/\log\log y$ in the
interval $[1,y]$ with $y=y_L=L^{1+\alpha}(\log L)^2$
(see~\cite[Theorem~328]{HarWr}). Summing up the
inequality~\eqref{eq:dd1d2} for all $d_0\le L^{1+\alpha}(\log L)^2$
and all $k_1,\ldots,k_{M-1}\le L^{1-\varepsilon}$, we get that the
number of $n \in \cI_{d,L}$ is at most
\begin{equation}
\label{eq:B11} \# \cI_{d,L} \ll   \frac{x_L(\log L)^{M+2}
L^{1+\alpha}L^{(M-1)(1-\varepsilon)}}{L^M}=\frac{x_L(\log
L)^{M+2}}{L^{(M-1)\varepsilon-\alpha}}.
\end{equation}
We now choose $M$ to be the least integer such that
$(M-1)\varepsilon>2+\alpha$, and with this choice of $M$ we get that
\begin{equation}
\label{eq:Ifinal} \# \cI_{d,L}\ll \frac{x_L}{L^2}.
\end{equation}
We now deal with the set $\cJ_{d,L}$ consisting of the numbers $n\in
\cH_{d,L}$ with $\#\cD(n)\le M^{-1}\(\log
n\)^{\alpha+\varepsilon}(\log\log n)^2$. To these, we apply
Lemma~\ref{lem:Div}. Since $\tau(n) < (\log n)^3$ and $P(2^n-1)<n^2$
for $n \in \cH_{d,L}$, Lemma~\ref{lem:Div} yields
$$
\log \Delta(n)/\log \log n \ll \#\cD(n) \ll \(\log
n\)^{\alpha+\varepsilon}(\log\log n)^2.
$$
Thus,
\begin{eqnarray*}
\log \Delta(n)  & \ll & \(\log n\)^{\alpha+\varepsilon}(\log\log n)^3 \\
& \ll & \(\log x_L\)^{\alpha+\varepsilon}(\log\log x_L)^3\ll
L^{\alpha+\varepsilon}(\log L)^3.
\end{eqnarray*}
Therefore
$$
\Delta_0(n)\le \Delta(n) \le z_L,
$$
where
$$
z_L=\exp(c  L^{\alpha+\varepsilon}(\log L)^3)
$$
and $c>0$ is some absolute constant.

We now further split $\cJ_{d,L}$ into two subsets. Let $\cS_{d,L}$
be the subset of $n\in \cJ_{d,L}$ such that $P(n)<x_L^{1/\log L}$.
 From known results concerning the distribution of smooth numbers
(see the corollary to Theorem~3.1 of~\cite{CEP}, or~\cite{HiTe},
\cite{Ten}, for example),
\begin{equation}\label{eq:S}
\#\cS_{d,L}\le \frac{x_L}{L^{(1+o(1))\log\log L}} \ll
\frac{x_L}{L^2}.
\end{equation}
Let $\cT_{d,L}=\cJ_{d,L}\backslash \cS_{d,L}$. For $n\in \cT_{d,L}$,
we have $n=qm$, where $q>x_L^{1/\log L}$ is a prime. Fix $m$. Then
$q<x_L/m$ is a prime such that $qdm+1$ is also a prime. By the Brun
sieve again,
\begin{equation}
\begin{split}
\label{eq:m} \#   \{ q \le x_L/m\ :  & \ q,
qdm+1\ \text{are primes}\}\\
& \ll  \frac{x_L}{m(\log(x_L/m))^2}\(\frac{md}{\varphi(md)}\) \ll
\frac{x_L (\log L)^3}{L^2 m},
\end{split}
\end{equation}
where in the above inequality we used the minimal order of the Euler
function in the interval $[1,x_LL^{1+\alpha}(\log L)^2]$ together
with the fact that
$$
\log(x_L/m) \ge \frac{\log x_L}{\log L}=\frac{L}{\log L}.
$$
We now sum up estimate~\eqref{eq:m} over all the allowable values
for $m$.

An immediate consequence of Lemma~\ref{lem:Struct} is that since
$\Delta_0(n)  \le z_L$, we also have $\Delta_0(m)  \le z_L$ for
$m=n/P(n)$. Thus, $m \in  \cG(x_L,z_L)$. Using Lemma~\ref{lem:saias}
and partial summation, we immediately get
\begin{align*}
\sum_{m\in \cG(x_L,z_L)}\frac{1}{m}
&\le \int_2^{x_L}\frac{d(\#\cG(t,z_L))}{t}
\le \frac{\#\cG(x_L,z_L)}{x_L}+
\int_2^{x_L}\frac{\#\cG(t,z_L)}{t^2} dt \\
&\ll \frac{\log z_L}{L} +\log
z_L\int_2^{x_L}\frac{dt}{t\log t}\\
& \ll   \log z_L \log\log x_L\ll L^{\alpha+\varepsilon} (\log L)^4,
\end{align*}
as $L\to\infty$. Thus,
\begin{equation}
\label{eq:T} \#\cT_{d,L}\ll \frac{x_L(\log L)^3}{L^2}\sum_{m\in
\cM_{d,L}} \frac{1}{m}\ll \frac{x_L(\log L)^7
L^{\alpha+\varepsilon}}{L^2}< \frac{x_L}{L^{2-\alpha-2\varepsilon}},
\end{equation}
when $L$ is sufficiently large. Combining
estimates~\eqref{eq:Ifinal}, \eqref{eq:S} and~\eqref{eq:T}, we get
that
\begin{equation}
\#\cH_{d,L}\le \#\cJ_{d,L}+\#\cS_{d,L}+\#\cT_{d,L}\ll
\frac{x_L}{L^{2-\alpha-2\varepsilon}}.
\end{equation}
Thus, returning to series~\eqref{eq:SL}, we get that
$$
S_L\le \sum_{d\le L^{1+\alpha}(\log
L)^2}\frac{1}{L^{2-2\alpha-2\varepsilon}}\ll \frac{\log
L}{L^{2-2\alpha-2\varepsilon}}.
$$
Since $\alpha<1/2$, we can choose $\varepsilon>0$ such that
$2-2\alpha-2\varepsilon>1$ and then the above arguments show that
$$
\sum_{n\ge 1}\frac{(\log n)^{\alpha}}{P(2^n-1)}\ll
1+\sum_{L}\frac{\log L}{L^{2-2\alpha-\varepsilon}}<\infty,
$$
which is the desired result.


\begin{thebibliography}{9999}

\bibitem{BV} G.~D.~Birkhoff and H.~S.~Vandiver, `On the integral
divisors of $a\sp n-b\sp n$', {\it Ann. of Math.\/} {\bf (2) 5}
(1904),   173--180.

\bibitem{CEP} E.~R.~Canfield, P.~Erd\H os and C.~Pomerance, `On a problem of
Oppenheim concerning ``factorisatio numerorum''', {\it J. Number
Theory\/} {\bf 17} (1983),  1--28.

\bibitem{EKP} P.~Erd\H os, P.~Kiss and C.~ Pomerance, `On prime divisors of
Mersenne numbers', {\it Acta Arith.\/} {\bf 57} (1991), 267--281.

\bibitem{ES} P.~Erd\H os and T.~N.~ Shorey, `On the greatest prime factor of
$2\sp{p}-1$ for a prime $p$ and other expressions', {\it Acta
Arith.\/} {\bf 30} (1976),  257--265.

\bibitem{HR} H. Halberstam and H.-E. Richert, {\it Sieve methods\/},
Academic Press, London, 1974.

\bibitem{HarRam} G.~H.~Hardy and S.~Ramanujan, `The normal number of
prime factors of an integer, {\it Quart. Journ. Math. (Oxford)\/}
{\bf 48} (1917), 76-92.

\bibitem{HarWr} G.~H.~Hardy and E.~M.~Wright,  \emph{An Introduction
to the Theory of Numbers}, 5th ed., Oxford, 1979.

\bibitem{HiTe} A.~Hildebrand and G.~Tenenbaum,
`Integers without large prime factors', {\it J.\ de Th{\'e}orie des
Nombres de Bordeaux\/}, {\bf 5} (1993), 411--484.

\bibitem{Matv} E.~M.~Matveev, `An explicit lower bound for a homogeneous
rational linear form in logarithms of algebraic numbers II', {\it
Izv. Ross. Akad. Nauk. Ser. Math.\/} {\bf 64} (2000), 125--180;
English translation {\it Izv. Math.\/} {\bf 64} (2000), 1217--1269.

\bibitem{MP} L.~Murata and C.~ Pomerance, `On the largest prime factor of a
Mersenne number', {\it Number theory CRM Proc. Lecture Notes
vol.36\/}, Amer. Math. Soc., Providence, RI, 2004, 209--218,.

\bibitem{MW} R.~Murty and S.~Wong, `The $ABC$ conjecture and prime
divisors of the Lucas and Lehmer sequences', {\it Number theory for
the millennium, III (Urbana, IL, 2000)\/},  A K Peters, Natick, MA,
2002, 43--54.

\bibitem{Pom1} C.~ Pomerance, `On primitive divisors of Mersenne
numbers', {\it Acta Arith.\/} {\bf 46} (1986), no. 4, 355--367.

\bibitem{Sa} E.~Saias, `Entiers \`a diviseurs denses $1$', {\it J.
Number Theory\/} {\bf 62} (1997), 163--191.

\bibitem{Sch} A. Schinzel, `On primitive prime factors of
$a\sp{n}-b\sp{n}$', {\it Proc. Cambridge Philos. Soc.\/} {\bf 58}
(1962), 555--562.

\bibitem{S1} C.~L.~Stewart, `The greatest prime factor of
$a\sp{n}-b\sp{n}$', {\it Acta Arith.\/} {\bf 26} (1974/75), no. 4,
427--433.

\bibitem{S2} C.~L.~Stewart, `On divisors of Fermat, Fibonacci, Lucas
and Lehmer numbers', {\it Proc. London Math. Soc.\/} (3) {\bf 35}
(1977), 425--447.

\bibitem{Ten}
G. Tenenbaum, {\it Introduction to analytic and probabilistic number
theory\/}, Cambridge Univ. Press, 1995.

\end{thebibliography}
\end{document}